\documentclass[a4paper,10pt]{amsart}
\usepackage{amsmath,amsthm,amssymb,latexsym,enumerate,color,hyperref}
\usepackage{graphicx}
\usepackage{xcolor}
\usepackage{eucal}
\numberwithin{equation}{section}

\begin{document}

\newtheorem{thm}{Theorem}[section]
\newtheorem{prop}[thm]{Proposition}
\newtheorem{lem}[thm]{Lemma}
\newtheorem{cor}[thm]{Corollary}
\newtheorem{ex}{Example}[section]      

\newtheorem{rem}[thm]{Remark}

\newtheorem*{defn}{Definition}

\newcommand{\DD}{\mathbb{D}}
\newcommand{\NN}{\mathbb{N}}
\newcommand{\ZZ}{\mathbb{Z}}
\newcommand{\QQ}{\mathbb{Q}}
\newcommand{\RR}{\mathbb{R}}
\newcommand{\CC}{\mathbb{C}}
\renewcommand{\SS}{\mathbb{S}}

\renewcommand{\theequation}{\arabic{section}.\arabic{equation}}

\newcommand{\usc}{\mathrm{USC}}
\newcommand{\lsc}{\mathrm{LSC}}

\newcommand{\Erfc}{\mathop{\mathrm{Erfc}}}    
\newcommand{\supp}{\mathop{\mathrm{supp}}}    
\newcommand{\re}{\mathop{\mathrm{Re}}}   
\newcommand{\im}{\mathop{\mathrm{Im}}}   
\newcommand{\dist}{\mathop{\mathrm{dist}}}  
\newcommand{\link}{\mathop{\circ\kern-.35em -}}
\newcommand{\spn}{\mathop{\mathrm{span}}}   
\newcommand{\ind}{\mathop{\mathrm{ind}}}   
\newcommand{\rank}{\mathop{\mathrm{rank}}}   
\newcommand{\ol}{\overline}
\newcommand{\pa}{\partial}
\newcommand{\ul}{\underline}
\newcommand{\diam}{\mathrm{diam}}
\newcommand{\lan}{\langle}
\newcommand{\ran}{\rangle}
\newcommand{\tr}{\mathop{\mathrm{tr}}}
\newcommand{\diag}{\mathop{\mathrm{diag}}}
\newcommand{\dv}{\mathop{\mathrm{div}}}
\newcommand{\na}{\nabla}
\newcommand{\nr}{\Vert}

\newcommand{\al}{\alpha}
\newcommand{\be}{\beta}
\newcommand{\ga}{\gamma}  
\newcommand{\Ga}{\Gamma}
\newcommand{\de}{\delta}
\newcommand{\De}{\Delta}
\newcommand{\ve}{\varepsilon}
\newcommand{\fhi}{\varphi} 
\newcommand{\la}{\lambda}
\newcommand{\La}{\Lambda}    
\newcommand{\ka}{\kappa}
\newcommand{\si}{\sigma}
\newcommand{\Si}{\Sigma}
\newcommand{\te}{\theta}
\newcommand{\zi}{\zeta}
\newcommand{\om}{\omega}
\newcommand{\Om}{\Omega}

\newcommand{\cA}{\mathcal{A}}
\newcommand{\cC}{\mathcal{C}}
\newcommand{\cE}{\mathcal{E}}
\newcommand{\cG}{{\mathcal G}}
\newcommand{\cH}{{\mathcal H}}
\newcommand{\cI}{{\mathcal I}}
\newcommand{\cJ}{{\mathcal J}}
\newcommand{\cK}{{\mathcal K}}
\newcommand{\cL}{{\mathcal L}}
\newcommand{\cM}{\mathcal{M}}
\newcommand{\cN}{{\mathcal N}}
\newcommand{\cP}{\mathcal{P}}
\newcommand{\cR}{{\mathcal R}}
\newcommand{\cS}{{\mathcal S}}
\newcommand{\cT}{{\mathcal T}}
\newcommand{\cU}{{\mathcal U}}
\newcommand{\cX}{\mathcal{X}}

\newcommand{\esssup}{\mathop{\mathrm{esssup}}}   
\newcommand{\essinf}{\mathop{\mathrm{essinf}}}   
\newcommand{\sgn}{\mathrm{sgn}}   

\newcommand{\loc} {\mathrm{loc}}

\title[Variational $p$-harmonious functions]{Variational $p$-harmonious functions:\\
 existence and convergence \\ to $p$-harmonic functions}

\author{E. W. Chandra}
\address{Department of Systems Innovation
Graduate School of Engineering Science, Osaka University,
1-3 Machikaneyama, Toyonaka, Osaka 560-8531, Japan}

\author{M. Ishiwata}
\address{Department of Systems Innovation
Graduate School of Engineering Science, Osaka University,
1-3 Machikaneyama, Toyonaka, Osaka 560-8531, Japan}
\email{ishiwata@sigmath.es.osaka-u.ac.jp}
\urladdr{}

\author{R. Magnanini}
\address{Dipartimento di Matematica e Informatica ``U. Dini'', Universit\`a di Firenze, viale Morgagni 67/A, 50134, Florence, Italy}
\email{rolando.magnanini@unifi.it}
\urladdr{}

\author{H. Wadade}
\address{Faculty of Mechanical Engineering, Institute of Science and Engineering, Kanazawa University,
Kakuma, Kanazawa, Ishikawa 920-1192, Japan}
\email{wadade@se.kanazawa-u.ac.jp}
\urladdr{}

\begin{abstract}
In a recent paper, the last three authors showed that a game-theoretic $p$-harmonic function $v$ is characterized by an asymptotic mean value property with respect to a kind of mean value $\nu_p^r[v](x)$ defined variationally on balls $B_r(x)$.
In this paper, in a domain $\Om\subset\RR^N$, $N\ge 2$, 
we consider the operator $\mu_p^\ve$, acting on continuous functions on $\ol{\Om}$, defined by the formula $\mu_p^\ve[v](x)=\nu^{r_\ve(x)}_p[v](x)$, where $r_\ve(x)=\min[\ve,\dist(x,\Ga)]$ and $\Ga$ denotes the boundary of $\Omega$.  
We first derive various properties of $\mu^\ve_p$ 
such as continuity and monotonicity. 
Then, we prove the existence and uniqueness 
of a function $u^\ve\in C(\ol{\Om})$ satisfying the Dirichlet-type problem:
$$
u(x)=\mu_p^\ve[u](x) \ \mbox{ for every } \ x\in\Om,\quad u=g \ \mbox{ on } \ \Ga,
$$
for any given function $g\in C(\Ga)$. This result holds, if we assume the existence of a suitable notion of barrier for all points in $\Ga$. That $u^\ve$ is what we call the \textit{variational} $p$-harmonious function with Dirichlet boundary data $g$, and
is obtained by means of a Perron-type method based on a comparison principle. 
\par
We then show that the family $\{ u^\ve\}_{\ve>0}$ gives an approximation scheme for the viscosity solution $u\in C(\ol{\Om})$ of
$$
\De_p^G u=0 \ \mbox{ in }\Om, \quad u=g \ \mbox{ on } \ \Ga,
$$
where $\De_p^G$ is the so-called game-theoretic (or homogeneous) $p$-Laplace operator.
In fact, we prove that $u^\ve$ converges to $u$, uniformly on $\ol{\Om}$ as $\ve\to 0$. 
\end{abstract}


\maketitle

\raggedbottom

\section{Introduction and main results}
\noindent

In this paper, we introduce and study what we call \textit{variationally $p$-harmonious functions}. Their definition is based on a variational notion of average for functions in a Lebesgue space. In fact, we define a \textit{$p$-mean} of a function $v\in L^p(E)$ as a real number $\nu_p[v]$ 
such that 
\begin{equation}
\label{def-p-mean}
\|v-\nu_p[v]\|_{p,E}=\min_{\nu\in\RR}\|v-\nu\|_{p,E}.
\end{equation}
Here, $L^p(E)$ denotes the standard Lebesgue space for $1\le p\le\infty$ (equipped with a suitably normalized norm), where $E$ is a Lebesgue measurable set in $\RR^N$ with finite measure. The existence and uniqueness of $\nu_p[v]$ is guaranteed for $1<p<\infty$, since it is the projection of $v$ on the subspace of constant functions.  Also,  $\nu_\infty[v]$ can be explicitly determined. When $p=1$, in order to uniquely determine $\nu_1[v]$, we need to assume that $v$ is continuous in an open set $E$. In general, $\nu_p[v]$ is nonlinear in $v$, unless $p$ equals $2$, in which case we recover the standard mean value of a function.
\par
If $\Om$ is a bounded domain in $\RR^N$, we can associate with $\nu_p$ the operator $\mu_p^\ve$ that acts on a function $v\in C(\ol{\Om})$ --- the space of continuous functions on $\ol{\Om}$ --- by the rule
$$
\mu_p^\ve[v](x)=\nu^{r_\ve(x)}_p[v](x) \ \mbox{ for } \ x\in\ol{\Om},
$$
where $\nu^r_p[v]$ denotes the $p$-mean of $v$, when we choose $E$ as the ball $B_r(x)$, centered at $x\in\RR^N$ and with radius $r>0$, and  $r_\ve(x)=\min[\ve,\dist(x,\Ga)]$ for $x\in\ol{\Om}$. Here,  $\dist(x,\Ga)$ denotes the distance of a point $x\in\ol{\Om}$ to the boundary $\Ga$ of $\Om$ (see Section \ref{subsec:p-mean} for details). 
The use of the function $r_\ve$ is inspired by the work in \cite{HR, HR2}, in which, for the planar case, it is associated to another family of nonlinear means.  
\par
By Theorem \ref{thm:semicontinuity} below, $\mu_p^\ve$ maps the space $C(\ol{\Om})$ into itself.
Thus, it is natural to study the set of functions $v\in C(\Om)$
satisfying the fixed-point equation:
$$
v=\mu_p^\ve[v] \ \mbox{ in } \ \Om.
$$ 
We call these functions
\textit{variationally $p$-harmonious} in $\Om$.
Affine functions are $p$-harmonious for any $p\in[1,\infty]$. Classical harmonic functions are $2$-harmonious. For $p\ne 2$, $p$-harmonic functions are not a subset of $p$-harmonious functions. Needless to say, functions that satisfy in $\Om$ the inequalities $v\le\mu_p^\ve[v]$ or $v\ge\mu_p^\ve[v]$ will be named in this paper \textit{variationally $p$-subharmonious} or \textit{$p$-superharmonious}.
\par
Harmonious functions have been introduced in \cite{LeG}, for the case $p=\infty$, in relationship with the so-called absolutely minimizing Lipschitz extensions. Classes of $p$-harmonious functions for a general $p\in(1,\infty]$ have been introduced in \cite{MPR0, MPR}, and in \cite{HR, HR2} also for the case $p=1$ in the plane, in connection with dynamic programming principles associated with \textit{tug-of-war} games (see also \cite{PS, PSSW}). 
We also refer the reader to \cite{AHP, CLM, FFGS, HR, HR2, LM, LMR, LP} for other related works that motivate these issues. 
\par
The classes of $p$-harmonious functions just mentioned are based on a definition of $p$-mean different from that given in \eqref{def-p-mean}. That mean, although defined explicitly, occasionally fails to fulfil certain desired properties, such as monotonicity and continuity.  Instead, the $p$-mean defined variationally in \eqref{def-p-mean}, whose study in the context of $p$-harmonious functions was introduced in \cite{IMW}, naturally enjoys those properties. In particular, as we shall see in Section \ref{sec:limits-harmonious-functions}, it satisfies the structural requirements, codified in \cite{TMP}, which yield the convergence of the underlying dynamic programming principles to $p$-harmonic functions.
\par
In Section \ref{sec:variationally-harmonious-functions}, we will show that variationally $p$-harmonious functions satisfy the weak and strong comparison principle, despite they are not solutions of a partial differential equation. 
These results are a by-product of the variational definition in \eqref{def-p-mean} and are corollaries of analogous ones derived in Section \ref{sec:p-means-and-operator} for the operator $\mu_p^\ve$.
\par
Once these local properties of variationally $p$-harmonious functions have been established, we consider the Dirichlet-type problem of finding a function $u=u^\ve$ in $C(\ol{\Om})$ such that
\begin{equation}
\label{dirichlet-harmonious}
u=\mu_p^\ve[u] \ \mbox{ in } \ \Om,\quad u=g \ \mbox{ on } \ \Ga.
\end{equation}
Here, $g$ is any given function in $C(\Ga)$. In order to obtain existence and uniqueness for this problem, we proceed in the wake of classical Perron's method for harmonic functions. In fact, we define two classes of continuous functions on $\ol{\Om}$, $\cS_g$ and $\cS^g$, which contain what we call the \textit{variational $p$-subsolutions and $p$-supersolutions} of problem \eqref{dirichlet-harmonious}, respectively. Precisely, $v\in\cS_g$ (resp. $v\in\cS^g$)  if and only if
$$
v\le\mu_p^\ve[v] \ \mbox{ (resp. $v\le\mu_p^\ve[v]$)} \ \mbox{ in } \ \Om,\quad v\le g \ \mbox{ (resp. $v\ge g$)} \ \mbox{ on } \ \Ga.
$$
The desired solution $u^\ve$ of \eqref{dirichlet-harmonious} is then obtained by checking that the functions defined by
$$
\ul{u}^\ve(x)=\sup_{v\in S_g}v(x), \quad
\ol{u}^\ve(x)=\inf_{w\in S^g}w(x) \ \mbox{ for } \ x\in\ol{\Om}
$$
coincide on $\ol{\Om}$, provided $\Ga$ satisfies some sufficient regularity assumption. Then, $u^\ve$ can be defined by $u^\ve\equiv \ul{u}^\ve=\ol{u}^\ve$ on $\ol{\Om}$, and the assumed regularity assumptions on $\Ga$ guarantee that $u^\ve$ satisfies the boundary condition in \eqref{dirichlet-harmonious}, for every given $g\in C(\Ga)$.
\par
In fact, similarly to what done in the classical case, the existence of a solution of \eqref{dirichlet-harmonious} is related to a suitable notion of barrier: given a point $x_0\in\Ga$, a \textit{barrier} at $x_0$ for the problem \eqref{dirichlet-harmonious}
is a continuous function $w=w^{x_0}$ on $\ol{\Om}$, which is variationally $p$-superharmonious in $\Om$, positive in $\ol{\Om}\setminus\{x_0\}$, and such that $w(x_0)=0$. If $x_0\in\Ga$ admits a barrier, we say that $x_0$ is a \textit{regular point} for the Dirichlet problem \eqref{dirichlet-harmonious}.
Our existence and uniqueness result then reads as follows.
\begin{thm}[Existence and uniqueness of $p$-harmonious functions]
\label{thm:existence-uniqueness-harmonious}
Let $\Om$ be a bounded domain in $\RR^N$ and take $\ve_0>0$ such that $\Om$ contains at least a ball of radius $\ve_0$. Then, for every $\ve\in(0,\ve_0]$, the Dirichlet-type problem \eqref{dirichlet-harmonious} admits a unique solution $u^\ve$ of class $C(\ol{\Om})$ 
for every $g\in C(\Ga)$ if and only if all points in $\Ga$ are regular for \eqref{dirichlet-harmonious}. 
\end{thm}
\noindent
In Corollary \ref{cor:barrier}, as it happens in the classical case, we will show that all points of $\Ga$ are regular if $\Ga$ satisfies the exterior sphere condition.
\par
If we compare Theorem \ref{thm:existence-uniqueness-harmonious} to similar results obtained in the literature (see \cite{HR2, LPS, MPR}), we can see some differences. Apart for the use of a different notion of $p$-mean, 
which is there given by the explicit definition
$$
\eta^\ve_p[v](x)=\frac{N+2}{N+p}\frac1{|B_\ve(x)|}\int_{B_\ve(x)}v\,dy
+\frac{p-2}{2\,(N+p)}\left[
\max_{
\ol{B_\ve(x)}
} v+\min_{
\ol{B_\ve(x)}
} v
\right]
$$
for $v\in C(\Om)$,
the role of the operator $\mu^\ve_p$ used in \eqref{dirichlet-harmonious} is replaced in \cite{LPS, MPR} by the mapping
$$
C(\Om_\ve)\ni v\mapsto \eta^\ve_p[v]\,\cX_\Om+g^*\,\cX_{\Ga_\ve}, 
$$
where $\Om_\ve=\Om\cup\Ga_\ve$, 
$\Ga_\ve=\{y\in\RR^N\setminus\Om : \dist(y,\Ga)\le\ve\}$ and $g^*$ is a continuous extension of the boundary data $g$ to $\Ga_\ve$. Thus, the solution of the relevant Dirichlet-type problem obtained with these ingredients is in general discontinuous on $\Ga$, differently from what stated in Theorem \ref{thm:existence-uniqueness-harmonious}. On the other hand, the construction of a $p$-harmonious function in those cases is obtained by purely monotonicity arguments, whereas we need the existence of a barrier. 
\par
The final purpose of this paper is to show that the operator $\mu_p^\ve$ 
gives an approximation scheme for the unique viscosity solution of the Dirichlet problem:
\begin{equation}
\label{dirichlet-harmonic}
\Delta_p^G u=0 \ \mbox{ in } \ \Om,\quad u=g \ \mbox{ on } \ \Ga.
\end{equation}
Here, $\De_p^G$ denotes 
{\it the game theoretic (or homogeneous) $p$-Laplacian}, (formally) defined by 
\begin{equation}
\label{game-theoretic-p-laplace}
\Delta^G_p u=\frac{\De u}{p}+\frac{p-2}{p}\,\frac{\lan \na^2u \na u, \na u\ran}{|\na u|^2}.
\end{equation}
For $1\le p<\infty$, we have that 
$$
\De^G_p u=\frac1{p}\,\frac{\dv(|\na u|^{p-2}\na u)}{|\na u|^{p-2}}. 
$$
The normalizing factor $p$ allows a formal definition in the limiting case $p=\infty$. With this choice, $\De^G_2=\De/2$, where $\De$ is the classical Laplace operator, and
$$
\De_\infty u=\frac{\lan \na^2u \na u, \na u\ran}{|\na u|^2}.
$$
\par
We observe that $\De_p^G$ is uniformly elliptic. However, it has discontinuous coefficients and is not variational. Thus, the fulfilment of \eqref{dirichlet-harmonic} must be intended in a suitable viscosity sense, which we will detail in Section \ref{sec:limits-harmonious-functions}. Notice that \textit{game-theoretic $p$-harmonic functions}, i.e. the viscosity solutions of the first equation in \eqref{dirichlet-harmonic}, have been proved to coincide with \textit{$p$-harmonic functions}, i.e. the weak solutions of the \textit{$p$-Laplace equation}, $\De_p u=0$, where
$$
\De_p u=\dv(|\na u|^{p-2}\na u).
$$
(See \cite{JLM}.) With these premises, we can state our convergence result.

\begin{thm}[Limits of $p$-harmonious functions]
\label{th:limit-of-harmonious}
Let $\Om$ be a bounded domain containing a ball of radius $\ve_0>0$. 
Suppose that there exists $\ve_\Ga\in (0, \ve_0]$ such that,
for every $0<\ve<\ve_\Ga$, $\Ga$ is made of regular points for the problem \eqref{dirichlet-harmonious}. 
\par
For $0<\ve<\ve_\Ga$, let $u^\ve$ be the unique solution in $C(\ol{\Om})$ of \eqref{dirichlet-harmonious}. Then, for every $g\in C(\Ga)$, there exists $u\in C(\ol{\Om})$ such that 
$u^\ve$ converges to $u$ uniformly on $\ol{\Om}$ as $\ve\to 0$ and $u$ is the unique viscosity solution of \eqref{dirichlet-harmonic}.
\end{thm}

Results similar to this theorem have been obtained in \cite{TMP, LPS, MPR}, based on the $p$-mean $\eta_p^\ve[v].$ In particular, this theorem was expected to hold 
since the average $\nu_p$ enjoys all the structural assumptions, 
proposed in \cite{TMP}, for the convergence of the underlying dinamic programming principle. Those assumptions are additivity with constants, $1$-homogeneity and monotonicity for essentially bounded functions (see Section \ref{subsec:p-mean}).
However, Theorem \ref{th:limit-of-harmonious} cannot be proved as a direct application of the general result contained in \cite{TMP}, since the relevant definitions of $p$-harmonious functions are different.  
Therefore, we adapt the argument used in \cite{TMP} to the case of the operator $\mu_p^\ve$ considered in this paper. We succeed in this goal with the help of an argument established in \cite{BS} and the proof of an asymptotic mean value property for $\mu_p^\ve$, which characterizes game-theoretic $p$-harmonic functions. This is based on the formula
\begin{equation}
\label{p-amvp}
\mu_p^\ve[\phi](x)=\phi(x)+\frac{p}{2(N+p)}\De_p^G\phi(x)\,\ve^2+o(\ve^2) \ \mbox{ as } \ \ve\to 0, 
\end{equation}
which holds for any $x\in\Om$ and any $\phi\in C^2(\Om)$ with $\na\phi(x)\ne 0$.
The formula \eqref{p-amvp} is obtained by adapting \cite[Theorem 3.2]{IMW}. 

\smallskip

As a final remark, we comment on the relationship between 
the barriers for \eqref{dirichlet-harmonious} and those for \eqref{dirichlet-harmonic}. 
By combining Theorem \ref{thm:existence-uniqueness-harmonious} 
with Theorem \ref{th:limit-of-harmonious}, we can observe that 
the existence of a barrier for \eqref{dirichlet-harmonious} for any $\varepsilon$ implicitly 
shows the existence of a barrier also for \eqref{dirichlet-harmonic} 
since the solvability of \eqref{dirichlet-harmonious} or \eqref{dirichlet-harmonic} 
is equivalent to the existence of a barrier in the corresponding sense 
together with the fact that the solvability of \eqref{dirichlet-harmonious} 
for any $\varepsilon$ implies the solvability of \eqref{dirichlet-harmonic}.

We summarize the contents of this paper. In Section \ref{sec:p-means-and-operator}, we present some old and new properties on the mean $\nu_p$ such as 
continuity and monotonicity of the operator $\mu_p^\ve$. In Section \ref{sec:variationally-harmonious-functions}, we introduce what we call variationally $p$-harmonious functions and show that they satisfy  weak and strong comparison principles. 
We shall prove existence and uniqueness for the Dirichlet problem for variationally $p$-harmonious functions in Section \ref{sec:existence-harmonious-functions}. Finally, in Section \ref{sec:limits-harmonious-functions}, we will show that limits of variationally $p$-harmonious functions are $p$-harmonic in a viscosity sense. 

\smallskip

\section{The $p$-mean and the operator $\mu_p^\ve$}
\label{sec:p-means-and-operator}
In this section, we collect some properties of the $p$-means that are instrumental in the proofs of Theorems \ref{thm:existence-uniqueness-harmonious} and \ref{th:limit-of-harmonious}. 

\subsection{Definition of $p$-mean and general properties}
\label{subsec:p-mean}
Let $E\subset\RR^N$ be a measurable set of finite measure. 
For a measurable function $v: E\to\ol{\RR}$, we shall use the notation:
\begin{eqnarray*}
&&\nr v\nr_{p,E}=\left(\frac1{|E|}\int_E |v|^p\,dx\right)^{1/p} \ \mbox{ for } \ 1\le p<\infty 
\ \mbox{ and } \\ 
&&\nr v\nr_{\infty,E}=\inf\{t\ge 0: |\{ x\in E: |v(x)|>t\}|\},
\end{eqnarray*}
whenever these quantities are finite.
Since $E$ has finite measure, we can (and will) always assume that $|E|=1$, without loss of generality.
\par
Let $v\in L^p(E)$. We define a \textit{$p$-mean} of $v$ as a real number $\nu_p[v]$ 
such that 
$$
\|v-\nu_p[v]\|_{p,E}=\min_{\nu\in\RR}\|v-\nu\|_{p,E}.
$$
As shown in \cite{IMW}, $\nu_p[v]$ is uniquely defined for $1<p\le\infty$ and is a natural generalization of the classical mean value of $v$. In fact, $\nu_2[v]$ is the average of $v$ on $E$. We also know that $\nu_1[v]$ may not be unique. For instance, if $E=[-2,2]$ and $v=\cX_{[-1,1]}$, then any number $\nu$ in $[0,1]$ fulfils the definition when $p=1$. 
However, $\nu_1[v]$ is uniquely defined if $v\in L^1(\Om)\cap C(\Om)$, when $\Om$ is an open domain of finite measure in $\RR^N$  (see \cite{No}). For the sake of brevity, from now on, in the case $p=1$ we shall assume that $v\in L^1(\Om)\cap C(\Om)$.
\par
It is useful to know that $\nu_p[v]$ is the unique root of the equation
\begin{equation}
\label{characterization}
\int_E |v-\nu|^{p-2}[v-\nu]\,dy=0,
\end{equation}
for $1\le p<\infty$, and
$$
\nu_\infty[v]=\frac12\left\{\esssup_E v+\essinf_E v\right\}
$$
(see \cite{IMW}).
For $p=1$, we mean that $|t|^{p-2} t=\sgn(t)$, which equals $1$ if $t>0$, $-1$ if $t<0$, and $0$ if $t=0$.  

\begin{rem}
{\rm
(i) It can be proved that $\nu_p[v]\to\nu_\infty[v]$ as $p\to\infty$. This fact can be obtained by taking the limit as $p\to\infty$ in the equality
$$
\left\{\int_{E_p^+}(v-\nu_p[v])^{p-1} dy\right\}^{\frac1{p-1}}=\left\{\int_{E_p^-}(\nu_p[v]-v)^{p-1} dy\right\}^{\frac1{p-1}},
$$
that follows from \eqref{characterization}. Here, $E_p^\pm=\bigl\{y\in E: v\gtreqless \nu_p[v]\bigr\}$.
\par
(ii)
Notice that while $\nu=\nu_p[v]$ is uniquely defined for any function $v$ in $L^p(E)$ with $1<p<\infty$, still, its characterization as solution of \eqref{characterization} should be analyzed in dependence on the values of $p$. In fact, notice that the integral 
$$
\int_E |v-\nu|^{p-2}\,dy
$$
is always finite for $2\le p<\infty$, but may not be so for $1<p<2$.
Thus, while for $2\le p<\infty$ the mean $\nu_p[v]$ can also be characterized as
$$
\nu_p[v]=\frac{\int_E |v-\nu_p[v]|^{p-2} v\,dy}{\int_E |v-\nu_p[v]|^{p-2}\,dy}
$$
(unless $v$ is a.e. constant on $E$), such a characterization may fail to be true for $1<p<2$, unless we extend it by setting $\nu_p[v]=0$, if the denominator is infinite. 
\par
For instance, if $E$ is the unit disk $B$ in $\RR^2$ and $v(x,y)=x^3-y^3$, by uniqueness we have that $\nu_p[v]=0$, since \eqref{characterization} is satisfied by this value for any  $1<p<\infty$. Nevertheless, we have that
$$
\int_B |v|^{p-2}\,dy=\infty \ \mbox{ and } \ \int_B |v|^{p-2} v\,dy=0 
\ \mbox{ for } \ 1<p<\frac43. 
$$
}
\end{rem}

\medskip

Elementary properties of $\nu_p[v]$ are the following (see \cite[Proposition 2.7 and Proposition 2.5]{IMW}): for every function $v\in L^p(E)$, it holds that
\begin{enumerate}[(a)]
\item
(additivity with constants)
$\nu_p[v+c]=\nu_p[v]+c$ for every  $c\in\RR$; 
\item
(homogeneity)
$\nu_p[\al v]=\al \nu_p[v]$ for every $\al\in\RR$;
\item
(monotonicity)
$\nu_p[v_1]\le \nu_p[v_2]$ if $v_1\le v_2$ a.e. in $E$. 
\end{enumerate}

For our purposes in this paper, we need to show that $\nu_p$ is \textit{strictly monotonic}.

\begin{lem}[Strict monotonicity]
\label{lem:strict-monotonicity}
Let $E\subset\RR^N$ be a measurable set with finite measure. 
Let $v_1, v_2\in L^p(E)$, $1\le p\le\infty$, and assume that $v_1\leq v_2$ a.e. in $E$ 
and that the measure of the set
$E^\sharp=\{x\in E :\, v_1(x)<v_2(x)\}$ is positive. 
Then it holds that $\nu_p[v_1]<\nu_p[v_2]$. 
In particular, if $v_1\leq v_2$ a.e. in $E$ 
and $\nu_p[v_1]=\nu_p[v_2]$, we have that $v_1=v_2$ a.e. in $E$. 
\end{lem}

\begin{proof}For $p=\infty$, our claim follows from an inspection. If $1\le p<\infty$,  the assumptions on $v_1$ and $v_2$ and the fact that 
the function $\RR\ni t\mapsto|t-\nu|^{p-2}(t-\nu)$ is strictly increasing
give that 
$$
\int_{E^\sharp}|v_1-\nu|^{p-2}(v_1-\nu) dy<\int_{E^\sharp}|v_2-\nu|^{p-2}(v_2-\nu) dy 
$$
for every $\nu$. Thus, the same holds with $E^\sharp$ replaced by $E$, and our claims follow from the characterization \eqref{characterization}.
\end{proof}

In \cite[Theorem 2.4]{IMW}, we proved the continuity of the operator $\nu_p$. 
That result can be relaxed.

\begin{lem}[Continuity]
\label{lem:continuity-projection}
Let  $E\subset\RR^N$ be a measurable set with finite measure and fix $1<p<\infty$. Let $v\in L^p(E)$  and  let $\{ v_n\}_{n\in\NN}$ be a sequence of measurable functions that converges to $v$ a.e. in $E$. 
\par
Suppose that either
\begin{enumerate}[(i)]
\item
$\{ v_n\}_{n\in\NN}$ is increasing and $v_n\ge 0$ on $E$ for every $n\in\NN$, or
\item
there exists a function $w\in L^p(E)$ such that $|v_n|\le w$ a.e. in $E$ for every $n\in\NN$.
\end{enumerate}
Then $\nu_p[v_n]\to \nu_p[v]$ as $n\to\infty$.
\end{lem}

\begin{proof}
(i) Set $\nu_n=\nu_p[v_n]$ and $\nu=\nu_p[v]$. The sequence of numbers $\nu_n$ converges or diverges to $+\infty$, since it increases. Let $\ol{\nu}$ be its limit. Since $v_n\le v$, we have that $\nu_n\le\nu$ and hence $\ol{\nu}\le\nu$. Now,  observe that
$$
0=\int_E |v_n-\nu_n|^{p-2}(v_n-\nu_n)\,dy\ge
\int_E |v_n-\ol{\nu}|^{p-2}(v_n-\ol{\nu})\,dy,
$$
since the function $\RR\ni s\mapsto |t-s|^{p-2}(t-s)$ is decreasing and $\nu_n\le\ol{\nu}$. 
The integrand at right-hand side of the last inequality is bounded from below by the number $-|\ol{\nu}|^{p-2} \ol{\nu}$. We thus can apply the Monotone Convergence Theorem and pass to the limit to obtain that
$$
\int_E |v-\nu|^{p-2}(v-\nu)\,dy=0\ge\int_E |v-\ol{\nu}|^{p-2}(v-\ol{\nu})\,dy.
$$
This inequality gives that $\ol{\nu}\ge\nu$, and the proof is complete.
\par
(ii)
The assumptions in (ii) give that $v_n\to v$ in $L^p(E)$ by the Dominated Convergence Theorem. Our claim then follows from \cite[Theorem 2.4]{IMW}.
\end{proof}

\begin{rem}
{\rm
When $p=\infty$, we obtain the convergence $\nu_\infty(v_n)\to\nu_\infty(v)$ as $n\to\infty$, if either $v_n\to v$ uniformly or, in the case (i), if $E$ is a bounded open set and $v_n, v\in C(E)$ for every $n\in\NN$. In fact, in this case the monotonicity of the sequence gives its uniform convergence, by Dini's monotone convergence theorem. \par
When $p=1$, in order to obtain the desired convergence, we must require that $v_n, v\in L^1(\Om)\cap C(\Om)$, where $\Om$ is an open domain of finite measure.
}
\end{rem}

\subsection{The operator $\mu^\ve_p$ and its properties}

Let $\Om\subset\RR^N$ be an open set. Let $B_r(x)$ be a ball of radius $r$, centered at $x$ and contained in $\Om$. We let $\nu^r_p[v](x)$ be the $p$-mean of $v$ relative to the set $E=B_r(x)$. In order to well-define $\nu^r_p[v](x)$, we shall always assume that 
$v\in L^p_\loc(\Om)$ for $1< p\le\infty$ and $v\in C(\Om)\cap L^1_{\loc}(\Om)$ when $p=1$.   
\par
Next, for each fixed $\ve>0$, we set $r_\ve(x)=\min[\ve, \dist(x,\Ga)]$ and define:
$$
\mu^\ve_p[v](x)=\nu^{r_\ve(x)}_p[v](x) \ \mbox{ for every } \ x\in\Om.
$$
We stress the fact that, differently from $\nu^r_p[v](x)$ that can be defined if $x$ is far enough from $\Ga$, $\mu^\ve_p[v](x)$ is well-defined for any $x\in\Om$. Notice that, from \eqref{characterization}, we easily infer that $\nu^r_p[v](x)$ is the unique root $\nu$ of the equation
\begin{equation}
\label{characterization-mu}
\int_B|v(x+r z)-\nu|^{p-2} [v(x+r z)-\nu]\,dz=0
\end{equation}
for $1\leq p<\infty$, and 
\begin{equation}
\label{characterization-mu-infty}
\nu_\infty^r[v](x)=\frac{1}{2}\left(
\operatorname{esssup}_{z\in B}v(x+r z)+
\operatorname{essinf}_{z\in B}v(x+r z)
\right), 
\end{equation}
where $B$ denotes the unit ball in $\RR^N$ centered at the origin. 
These facts tell us that we can assume that $\mu^\ve_p[v]=v$ on $\Ga$, at least formally.  In fact, for $x\in\Ga$,  $r_\ve(x)=0$ and \eqref{characterization-mu}-\eqref{characterization-mu-infty} give that $\nu=v(x)$ when $r=0$. This issue will be discussed in more detail in Theorem \ref{thm:semicontinuity} below.
\par
The following result extends a formula obtained in \cite[Theorem 3.2]{IMW}.
\begin{prop}
\label{uni-amvp}
Let $\Om$ be a domain in $\RR^N$ and let $\phi\in C^2(\Om)$. Then, for every bounded open set $A$ with $\ol{A}\subset\Om$ and such that $\nabla\phi\ne 0$ on $\ol{A}$, it holds that 
$$
\lim_{\ve\to 0}\frac{\mu_p^\ve[\phi]-\phi}{r_\ve^2}
=\frac{p}{2(N+p)}\,\De^G_p \phi \ \mbox{ uniformly on } \ \ol{A}. 
$$
\end{prop}

\begin{proof}
Let $x_\ve$ a point in $\ol{A}$ maximizing the difference 
$$
\left|\frac{\mu_p^\ve[\phi]-\phi}{r_\ve^2}
-\frac{p}{2(N+p)}\,\De^G_p \phi\right|
$$
on $\ol{A}$. Up to a subsequence, we have that $x_\ve\to x$ for some $x\in\ol{A}$, and hence $r_\ve(x_\ve)\to 0$ as $\ve\to 0$. The conclusion then follows from \cite[Theorem 3.2]{IMW}. 
\end{proof}

The following claim is a straightforward consequence of Lemma \ref{lem:strict-monotonicity}.

\begin{cor}[Monotonicity]
\label{cor:mean-bound}
Let $\Om\subset\RR^N$ be a domain of finite measure.  Let $v\in L^p(\Om)$, $1\le p<\infty$, $v\in L^1(\Om)\cap C(\Om)$ for $p=1$. 
Let $v_1, v_2\in L^p(\Om)$ be such that $v_1\le v_2$ a.e. in $\Om$. Then $\mu_p^\ve[v_1]\le\mu_p^\ve[v_2]$ on $\ol{\Om}$. Moreover, we have that $\mu_p^\ve[v_1](x)<\mu_p^\ve[v_2](x)$ whenever the set 
$$
\{y\in B_{r_\ve(x)}(x): v_1(y)<v_2(y)\}
$$ 
has positive measure.
\end{cor}

Next, we show that $\mu_p^\ve$ acts quite naturally on $\usc(\ol{\Om})$ and $\lsc(\ol{\Om})$,
the classes of functions that are upper and lower semicontinuous on $\ol{\Om}$, respectively. We state our result for the case of lower semicontinuous functions. The remaining case will easily follow.

\begin{thm}[Invariance semicontinuity property]
\label{thm:semicontinuity}
Let $\Om\subset\RR^N$ be a bounded open set. 
If $v\in\lsc(\ol{\Om})$, then we have that $\mu^\ve_p[v]\in\lsc(\ol{\Om})\cap  C(\Om)$.
\end{thm}

\begin{proof}
We first prove that $\mu^\ve_p[v]\in C(\Om)$. We take a point $x\in\Om$ and any sequence of points $x_n\in\Om$ converging to $x$. Then, we observe that 
$$
\mu_p^\ve[v](x)=\mu_p^1[w^x](0), 
$$
where $w^x : B\to\RR$ is defined by 
\begin{equation}
\label{def-w-x}
w^x(z)=v(x+r_\ve(x)z)\ \text{ for } \ z\in B. 
\end{equation}
In other words, $\mu_p^\ve[v](x)=\nu_p[w^x]$, when we choose $E$ to be the unit ball $B$. 
\par
Being as $v\in L^p_\loc(\Om)$ and $r_\ve$ continuous, we infer that $w^{x_n}\to w^x$ in $L^p(B)$ as $n\to\infty$. Thus, \cite[Theorem 2.4]{IMW} gives that $\nu_p[w^{x_n}]\to \nu_p[w^x]$, and hence we conclude that
$\mu_p^\ve[v](x_n)\to \mu_p^\ve[v](x)$. 
\par
We are left to prove the desired semicontinuity at points in $\Ga$. Thus, we take any sequence of points $x_n\in\ol{\Om}$ converging to some point $x\in\Ga$.
Observe that, since $\ol{\Om}$ is compact, $v$ admits its minimum $m>-\infty$ on it. Next, we proceed as before and consider the function $w^{x_n}$ in \eqref{def-w-x}. This is well-defined at points $x_n\in\Om$. At points $x_n\in\Ga$, instead, we consistently mean that $w^{x_n}=v(x_n)$. Since $\nu_p[w^{x_n}]=\nu_p[w^{x_n}-m]+m$ and $w^{x_n}\ge m$, we can always suppose that $w^{x_n}\ge 0$.
By the semicontinuity of $v$ and the continuity of $r_\ve$, we know that 
$$
\liminf_{n\to\infty} w^{x_n}(z)\ge w^x(z)
\ \mbox{ for any } \ z\in B. 
$$
Now, we have that 
$$
\inf_{k\geq n} \left(w^{x_k}\right)\to \liminf_{n\to\infty}\left(w^{x_n}\right)\ \mbox{ on $B$ as } \ n\to\infty,
$$
and the convergence is monotone increasing. 
Thus, we can apply Lemma \ref{lem:continuity-projection} (i) to obtain that 
$$
\lim_{n\to\infty} \nu_p\!\left[\inf_{k\geq n} \left(w^{x_k}\right)\right]= \nu_p\!\left[\liminf_{n\to\infty}\left(w^{x_n}\right)\right].
$$
Notice that this holds even if there are infinitely many terms $x_n$ that belong to $\Ga$.
\par
Therefore, we conclude that
\begin{multline*}
\liminf_{n\to\infty} \mu_p^\ve[v](x_n)=\liminf_{n\to\infty} \nu_p[w^{x_n}]\ge 
\liminf_{n\to\infty} \nu_p\!\left[\inf_{k\geq n} \left(w^{x_k}\right)\right]= \\
\nu_p\!\left[\liminf_{n\to\infty}\left(w^{x_n}\right)\right]\ge \nu_p[w^x]=\mu_p^\ve[v](x).
\end{multline*}
This gives that $\mu_p^\ve[v]$ is lower semicontinuous on $\ol{\Om}$.
\end{proof}

\smallskip

\section{Variationally $p$-harmonious functions}
\label{sec:variationally-harmonious-functions}
Slightly differently from what done in \cite{TMP, MPR}, we give the following definitions. 
Let $\Om\subset\RR^N$ be a bounded domain. Set $1\le p\le\infty$ and fix $\ve>0$ such that $\Om$ contains at least one ball of radius $\ve$. We say that $v\in L^p(\Om)$ is a {\it variationally $p$-subharmonious (or variationally $p$-superharmonious)} function in $\Om$ if
$$
v(x)\le\mu^\ve_p[v](x) \quad \mbox{(or $v(x)\ge\mu^\ve_p[v](x)$) \ for almost every  \ $x\in \Om$.}
$$ 
A \textit{variationally $p$-harmonious} function is both variationally $p$-subharmonious and $p$-superharmonious. We stress the fact that the term ``variationally'' refers to the variational definition of the operator $\mu^\ve_p$.

\begin{ex}
\label{ex:examples}
{\rm
(i) It is easily seen that constant functions are variationally $p$-harmonious for any $p\in [1,\infty]$. 
\par
(ii) For $\xi\in\RR^N$, let $a(y)=\lan\xi, y\ran+c$ be an affine function. We have that $a(y)=a_0(y)+\lan\xi, x\ran+c$, where $a_0(y)=\lan\xi, y-x\ran$ and $\lan\xi, x\ran+c$ is constant. It is easy to infer that  $\nu_p^r[a_0](x)=0$, by the central symmetry of $a_0$ with respect to $x$. Thus, we conclude that 
$$
\nu_p^r[a](x)=\nu_p^r[a_0](x)+\nu_p^r[\lan\xi, x\ran+c](x)= \lan\xi, x\ran+c=a(x).
$$
It is then clear that $\mu_p^\ve[a]=a$.
\par
Therefore, we infer that affine functions are always variationally $p$-harmonious. It is then evident that convex and concave functions are variationally $p$-subharmonious and  $p$-superharmonious, respectively.
\par
(iii) Fix an $\al>0$ and set $\ga_\al(y)=|y|^{-\al}$ for $y\ne 0$. Clearly, we have that $\ga_\al\in L^p_\loc(\RR^N)$ if $\al p<N$ and $\ga_\al\in L^p_\loc(\RR^N\setminus\{0\})$ for any $p\in[1,\infty]$. 
\par
Next, let 
$$
\De_p^G \phi=\frac1{p}\,\left\{\De \phi+(p-2)\,\frac{\lan\na^2 \phi \na \phi, \na \phi\ran}{|\na \phi|^2}\right\}
$$
denote the \textit{game-theoretic $p$-laplacian} of a smooth function away from its critical points.
We easily compute that
$$
\De^G_p\ga_\al(y)=\frac{\al\,[\al (p-1)+p-N]}{p}\,|y|^{-(\al+2)} \ \mbox{ for } \ y\ne 0.
$$
\par
By applying \cite[formula (3.6) in Theorem 3.2]{IMW},  for any $x\ne 0$, we obtain that
$$
\frac{\ga_\al(x)-\nu^r_p[\ga_\al](x)}{r^2}=-\frac{\al[\al (p-1)+p-N]}{2(N+p)}\,|x|^{-(\al+2)} +o(1)
$$
as $r\to 0$. 
\par
Now, let  $\Om$ be a bounded open set such that $\ol{\Om}$ does not contain the origin.
Depending on the sign of $\al (p-1)+p-N$, there exists $r_\Om>0$ such that the inequalities $\ga_\al\ge\nu_p^r[\ga_\al]$ or $\ga_\al\le\nu^r_p[\ga_\al]$ are uniformly satisfied in $\ol{\Om}$ for $0<r<r_\Om$. 
Therefore, for $0<r<r_\Om$,  $\ga_\al$ is variationally $p$-subharmonious in $\Om$, if 
$$
p> \frac{\al+N}{\al+1},
$$
and variationally $p$-superharmonious in $\Om$, if 
$$
p< \frac{\al+N}{\al+1}.
$$
}
\end{ex}

\medskip

\subsection{Comparison principles}
Despite they are not solutions of any partial differential equation, variationally $p$-harmonious functions satisfy comparison principles.

\begin{thm}[Weak comparison principle]
\label{th:weak-comparison}
Let $\Om\subset\RR^N$ be a bounded open set. Let $v\in\usc(\ol{\Om})$ and $w\in\lsc(\ol{\Om})$ be $p$-subharmonious and $p$-superharmonious in $\Om$, respectively.
\par
If $v\le w$ on $\Ga$, then $v\le w$ on $\ol{\Om}$.
\end{thm}

\begin{proof}
We proceed by contradiction. We have that $v-w\in\usc(\ol{\Om})$, and hence it attains its maximum $M$ on $\ol{\Om}$.  Suppose that $M$
is positive.  Since $v\le w$ on $\Ga$, the set 
$A=\{x\in\Om : v(x)-w(x)=M\}$ contains at least one point $x_0$. Let $\Om'$ be the connected component of $\Om$ that contains $x_0$. The set $A'=\{x\in\Om' : v(x)-w(x)=M\}$ is thus non-empty and also closed in $\Om'$, being the pre-image of the closed set $\{M\}\subset\RR$ under the continuous function $v-w:\Om'\to\RR$.
\par
To show that $A'$ is open, we take any $x\in A'$ and observe that 
$$
\mu_p^\ve[v-M](x)\le\mu_p^\ve[w](x)\le w(x)
$$
since $v-M\le w$ in $\Om$ by construction and $w$ is $p$-superharmonious. Thus, we infer that
$$
w(x)\ge\mu_p^\ve[v-M](x)=\mu_p^\ve[v](x)-M\ge v(x)-M=w(x) 
$$
since $v$ is $p$-subharmonious.
In particular, $\mu_p^\ve[w](x)=\mu_p^\ve[v-M](x)$, 
which gives that $w\equiv v-M$ on $B_{r_\ve(x)}(x)$, by Lemma \ref{lem:strict-monotonicity}. In other words, 
we have that $B_{r_\ve(x)}(x)$ is contained in $A'$. Thus, we have proved that $A'$ is also open and, since $\Om'$ is connected, we have that $A'=\Om'$. Thus, $v-w\equiv M$ on $\Om'$ and, by continuity, we infer that $M\equiv v-w$ on the boundary of $\Om'$, which is contained in $\Ga$. Therefore, we reach the contradiction $0<M\leq 0$.
\end{proof}

By more or less the same arguments, we can prove the following strong comparison principle.

\begin{cor}[Strong comparison principle]
Let $\Om\subset\RR^N$ be a bounded domain. Let $v, w\in C(\Om)$ be $p$-subharmonious and $p$-superharmonious in $\Om$, respectively.
\par
Suppose that $v\le w$ in $\Om$. Then, either $v<w$ in $\Om$ or, else, $v\equiv w$ in $\Om$.
\end{cor}

\begin{proof}
Suppose that $v-w=0$ at some point in $\Om$. Then, we repeat the same argument in  the proof of Theorem \ref{th:weak-comparison} with $A'=\{ x\in\Om: v(x)-w(x)=0\}$ and $M=0$. We thus infer that $A'$ is non-empty, closed, and open. Since $\Om$ is connected, we then obtain that $A'=\Om$, and hence $v\equiv w$ in $\Om$.
\end{proof}

\smallskip

\section{Existence of $p$-harmonious functions}
\label{sec:existence-harmonious-functions}
\noindent

In this section, we shall prove Theorem \ref{thm:existence-uniqueness-harmonious}. We start with some definitions.
\par
Let $\Om$ be a bounded domain in $\RR^N$ with boundary $\Ga$ and let $g:\Ga\to\RR$ be a continuous function. For a fixed number $\ve>0$, we introduce two classes of functions:
\begin{eqnarray*}
&&\cS_g=\{v\in C(\ol{\Om}) :\,v\le\mu_p^\ve[v] \ \mbox{ in }\ \Om,
\ v\le g \ \mbox{ on } \ \Ga\}, \\
&&\cS^g=\{v\in C(\ol{\Om}) :\,v\ge\mu_p^\ve[v] \ \mbox{ in }\ \Om,
\ v\ge g \ \mbox{ on } \ \Ga\}. 
\end{eqnarray*}
It is clear that $\cS^g=-\cS_{-g}$. The elements of $\cS_g$ and $\cS^g$ will be called \textit{variational $p$-subsolutions} and, respectively, \textit{$p$-supersolutions} of \eqref{dirichlet-harmonious}. These classes are non-empty since they contain the constant functions on $\ol{\Om}$ defined by
$$
v\equiv\min_\Ga g \ \mbox{ and } \ w\equiv\max_\Ga g,
$$
respectively.
\par
Our aim is to obtain a solution of \eqref{dirichlet-harmonious} as an envelope of functions in $\cS_g$ or in $\cS^g$. In fact, we define the \textit{lower and upper Perron solutions} $\ul{u}^\ve$ and $\ol{u}^\ve$ of \eqref{dirichlet-harmonious} by 
$$
\ul{u}^\ve(x)=\sup_{v\in \cS_g}v(x) \ \mbox{ and } \ 
\ol{u}^\ve(x)=\inf_{w\in \cS^g}w(x) \ \mbox{ for } \ x\in\ol{\Om}.
$$
The following lemma collects some elementary properties of these functions.

\begin{lem}
\label{lem3}
Let $\Om$ be a bounded domain in $\RR^N$ and let $g\in C(\Ga)$. Then,
\begin{enumerate}[(i)]
\item 
 it holds that
$$
\min_\Ga g\le\ul{u}^\ve\le\ol{u}^\ve\le\max_{\Ga} g \ \mbox{ on } \ \ol{\Om};
$$

\item
$u\in C(\ol{\Om})$ is a solution of  \eqref{dirichlet-harmonious} 
if and only if $u\in \cS_g\cap \cS^g$;

\item
it holds that $u=\ul{u}^\ve=\ol{u}^\ve$ on $\ol{\Om}$, if $u\in C(\ol{\Om})$ is a solution of  \eqref{dirichlet-harmonious}. 
\end{enumerate}
\end{lem}

\begin{proof}
(i) By the weak comparison principle (Theorem \ref{th:weak-comparison}), we have that $v\leq w$ on $\ol{\Om}$ for every $v\in \cS_g$ and $w\in \cS^g$, being as $v\le g\le w$ on $\Ga$. Thus, $\ul{u}^\ve\le\ol{u}^\ve$, clearly, by the properties of supremum and infimum. The other two inequalities easily follow from the fact that the two relevant constant functions belong to the classes $\cS_g$ or $\cS_g$.
\par
(ii) 
The statement is trivial by the definitions of $\cS_g$ and $\cS^g$. 
\par
(iii) If $u\in C(\ol{\Om})$ is a solution of \eqref{dirichlet-harmonious}, then (ii) gives that
$u\in \cS_g\cap \cS^g$, and hence $\ol{u}^\ve\le u\le \ul{u}^\ve\le\ol{u}^\ve$ on $\ol{\Om}$.
\end{proof}

In analogy with the classical case and as already mentioned in the introduction, we shall say that
a function $w$ is a \textit{barrier} at a point $x_0\in\Ga$ for the Dirichlet-type problem \eqref{dirichlet-harmonious} if
$w\in C(\ol{\Om})$, 
$w$ is superharmonious in $\Om$,  
and
$w>0$ in $\ol{\Om}\setminus\{x_0\}$ with $w(x_0)=0$. 
If $x_0\in\Ga$ admits a barrier, we say that $x_0$ is  a \textit{regular point} for \eqref{dirichlet-harmonious}.

\begin{prop}
\label{prop:barrier-result1}
If $x_0\in\Ga$ is a regular point for $\eqref{dirichlet-harmonious}$,
then 
$$
\ul{u}^\ve(x_0)=\ol{u}^\ve(x_0)=g(x_0).
$$
\end{prop}

\begin{proof}
The proof proceeds as in the classical case. Let $\eta>0$. Since $g\in C(\Ga)$, we can find $\de>0$ such that
$|g-g(x_0)|<\eta$ on $\Ga\cap B_\de(x_0)$. Also, if $w$ is a barrier at $x_0$, 
the following number is well-defined:
$$
M_\de=\max_{\Ga\setminus B_\de(x_0)}\frac{|g-g(x_0)|}{w}.
$$
Thus, we infer that
$$
|g-g(x_0)|\le \eta+M_\de\, w \ \mbox{ on } \ \Ga.
$$
This fact gives that the functions $g(x_0)+\eta+M_\de\,w$ and $g(x_0)-\eta-M_\de\,w$ belong to the classes $\cS^g$ and $\cS_g$, respectively. Hence, by definition
$$
g(x_0)-\eta-M_\de\,w\le \ul{u}^\ve\le\ol{u}^\ve\le g(x_0)+\eta+M_\de\,w \ \mbox{ on } \ \ol{\Om}
$$
and, in particular
$$
g(x_0)-\eta\le\ul{u}^\ve(x_0)\le\ol{u}^\ve(x_0)\le g(x_0)+\eta,
$$
being as $w(x_0)=0$. The desired conclusion then ensues, since $\eta>0$ is arbitrary.
\end{proof}

\begin{proof}[\bf Proof of Theorem \ref{thm:existence-uniqueness-harmonious}]
(i) Let us assume that every point in $\Ga$ admits a barrier for \eqref{dirichlet-harmonious}.
Lemma \ref{lem3} and the definition of $\ul{u}^\ve$ give that $\ul{u}^\ve$ is bounded and lower semicontinuous on $\ol{\Om}$, that is $\ul{u}^\ve\in L^\infty(\ol{\Om})\cap\lsc(\ol{\Om})$.
The assumption also gives that $\ul{u}^\ve=g$ on $\Ga$, thanks to Proposition \ref{prop:barrier-result1}. 
\par
Next, we define a sequence of functions by the following iterative scheme:
$$
u_1=\ul{u}^\ve, \quad
u_{j+1}=\mu_p^\ve[u_j] \ \mbox{ on } \ \ol{\Om}, \mbox{ for } j\geq 1. 
$$
Proposition \ref{thm:semicontinuity} ensures that all the functions $u_j$ belong to $L^\infty(\ol{\Om})\cap\lsc(\ol{\Om})$.
Clearly, it holds that $u_j=g$ on $\Ga$ for every $j\geq 1$. 
\par
Now, since $v\le\ul{u}^\ve$ on $\ol{\Om}$ for every $v\in \cS_g$, 
we have that $v\le\mu_p^\ve[v]\le\mu_p^\ve[\ul{u}^\ve]$ in $\Om$ 
for every $v\in \cS_g$, which gives $\underline u^\ve\leq\mu_p^\ve[\underline u^\ve]$ in $\Omega$. Hence, we can infer that $u_1\le\mu_p^\ve[u_1]=u_2$ on $\ol{\Om}$. 
Then by monotonicity, we have that $u_2=\mu_p^\ve[u_1]\le\mu_p^\ve[u_2]=u_3$ and, by iterating,  
we obtain that $u_j\leq u_{j+1}$ on $\ol{\Om}$ for every $j\geq 1$. 
Thus, the sequence $\{u_j\}_{j\in\NN}$ is increasing on $\ol{\Om}$. By iteration, we also have that
$$
\min_\Ga g\le u_j\le \max_\Ga g \ \mbox{ on } \ \ol{\Om}, \ \mbox{ for every } \ j\ge 1,
$$
because, by Lemma \ref{lem3} (i), we see that 
$$
\min_\Ga g\le u_1=\ul{u}^\ve\le\ol{u}^\ve\le\max_{\Ga} g \ \mbox{ on } \ol{\Om}.
$$
Thus, the sequence $\{u_j\}_{j\in\NN}$ converges pointwise on $\ol{\Om}$ to a function $u$, which is lower semicontinuous on $\ol{\Om}$ and such that 
$$
\min_\Ga g\le u\le\max_{\Ga} g \ \mbox{ on } \ol{\Om} \ \mbox{ and } \ u=g \ \mbox{ on } \ \Ga.
$$
Moreover, 
applying Lemma \ref{lem:continuity-projection} gives that
$\mu_p^\ve[u_j]\to\mu_p^\ve[u]$ as $j\to\infty$ in $\Om$, 
and hence $u=\mu_p^\ve[u]$ in $\Om$, since $u_j=\mu_p^\ve[u_j]$ in $\Om$, for every $j\ge 1$. As a result, the function $u$ satisfies the Dirichlet problem \eqref{dirichlet-harmonious}.
\par
By playing on the fact that $\cS^g=-\cS_{-g}$, we can determine a decreasing sequence of functions $U_j$, bounded and upper semicontinuous on $\ol{\Om}$, which converge to another solution $U$ of \eqref{dirichlet-harmonious}. Clearly, such a sequence is initialized by choosing $U_1=\ol{u}^\ve$, so that, by comparison, the relevant iterating schemes tell us that $U\ge u$ on $\ol{\Om}$, being as $U_1=\ol{u}^\ve\ge\ul{u}^\ve=u_1$ on $\ol{\Om}$.
\par
Thus, we are left to prove that $u=U$ in $\ol{\Om}$. Notice first that $U-u$ is non-negative and belongs to $\usc(\ol{\Om})$, being as $U\in\usc(\ol{\Om})$ and $u\in\lsc(\ol{\Om})$.
Hence, there exists $x_0\in\ol{\Om}$ such that 
$$
\max_{\ol{\Om}}(U-u)=(U-u)(x_0).
$$ 
Our goal is to prove that $(U-u)(x_0)=0$. 
\par
By contradiction, we assume that $(U-u)(x_0)>0$ and consider the set 
$$
\Om^\sharp=\bigl\{x\in\ol{\Om}: (U-u)(x)=(U-u)(x_0)>0\bigr\}.
$$
Since $U-u=g-g=0$ on $\Ga$, we have that $\Om^\sharp$ is a subset of $\Om$.
Theorem \ref{thm:semicontinuity} then tells us that $U-u$ is continuous in $\Om$, and hence we can infer that $\Om^\sharp$ is closed in $\Om$, because it is the pre-image of the singleton $\{ (U-u)(x_0)\}\subset\RR$.
\par
Next, notice that $U-U(x_0)\le u-u(x_0)$ in $\Om$, since $U-u\le (U-u)(x_0)$ in $\Om$, so that, if we take a point $x\in\Om$, we get that
$$
\mu_p^\ve[U](x)-U(x_0)=\mu_p^\ve[U-U(x_0)](x)\le \mu_p^\ve[u-u(x_0)](x)=\mu_p^\ve[u](x)-u(x_0),
$$
by the monotonicity of $\mu_p^\ve$. Thus, we infer that
$$
U(x)-U(x_0)=\mu_p^\ve[U-U(x_0)](x)\le \mu_p^\ve[u-u(x_0)](x)=u(x)-u(x_0),
$$
since both $U$ and $u$ are variationally $p$-harmonious in $\Om$.
\par
Now, if $x\in\Om^\sharp$, we know that $u(x)-u(x_0)=U(x)-U(x_0)$, and hence we obtain that
$$
U(x)-U(x_0)=\mu_p^\ve[U-U(x_0)](x)=\mu_p^\ve[u-u(x_0)](x)=U(x)-U(x_0).
$$
The definition of $\mu_p^\ve$ and Lemma \ref{lem:strict-monotonicity} then tell us that 
$U-U(x_0)\equiv u-u(x_0)$, and hence $U-u\equiv (U-u)(x_0)$ on the ball $B_r(x)$ with $r=r_\ve(x)>0$. As a consequence, we have shown that $B_r(x)\subset\Om^\sharp$, that is $\Om^\sharp$ is open, since $x$ is arbitrarily chosen in $\Om^\sharp$. All in all, $\Om^\sharp$ is a closed and open subset of $\Om$ and, moreover, is non-empty, since it contains at least $x_0$. Therefore, $\Om^\sharp=\Om$, because $\Om$ is connected, that is $U-u\equiv (U-u)(x_0)$ in $\Om$. 
\par
In conclusion, we find a contradiction by taking a sequence of points $x_n\in\Om$ that converge to any point $\ol{x}\in\Ga$. In fact, since $U-u\in\lsc(\ol{\Om})$, we have that
$$
0<(U-u)(x_0)=\limsup_{n\to\infty} (U-u)(x_n)\le (U-u)(\ol{x})=g(\ol{x})-g(\ol{x})=0.
$$
\end{proof}

\begin{cor}
\label{cor:any-initialization}
Let $w:\ol{\Om}\to\RR$ be any function such that  $\ul{u}^\ve\le w\le\ol{u}^\ve$ on $\ol{\Om}$. Consider the sequence of functions defined by
$$
w_1=w, \quad
w_{j+1}=\mu_p^\ve[w_j] \ \mbox{ on } \ \ol{\Om} \ \mbox{ for } j=1, 2, \cdots. 
$$
Then, we have that 
$$
\lim_{j\to\infty} w_j\to u^\ve=\ul{u}^\ve=\ol{u}^\ve \ \mbox{ on } \ \ol{\Om}.
$$
\end{cor}

\begin{proof}
By comparison, we can infer that
$$
u_j\le w_j\le U_j \ \mbox{ on } \ \ol{\Om} 
$$
for any $j\in\NN$. Thus, we have that
$$
u\le\liminf_{j\to\infty} w_j\le\limsup_{j\to\infty} w_j\le U=u,
$$
which gives the desired claim.
\end{proof}

\begin{cor}
\label{cor:barrier}
Let $\Om$ be a bounded domain in $\RR^N$, satisfying the uniform exterior sphere condition at any point in $\Ga$. Then there exists $0<\ve_\Ga\le\ve_0$ such that 
for every $0<\ve<\ve_\Ga$, all points in $\Ga$ are regular for \eqref{dirichlet-harmonious}.
\par
In particular, for every $g\in C(\Ga)$, the Dirichlet-type problem \eqref{dirichlet-harmonious} admits the unique solution $u^\ve=\ul{u}^\ve=\ol{u}^\ve\in C(\ol{\Om})$ defined by Perron's method. 
\end{cor}

\begin{proof}
Our assumption is that there exists s radius $R>0$ such that, for every $x_0\in\Ga$, we can find a ball
$B_R(y_0)$ such that $\ol{B_R(y_0)}\cap\ol{\Om}=\{x_0\}$. Hence, we define a function $w$ by 
$$
w(x)=\frac{1}{R^\al}-\frac{1}{|x-y_0|^\al} \ \mbox{ for } \ x\in\ol{\Om}. 
$$
\par
It is clear that $w\in C(\ol{\Om})$, $w$ is positive in $\ol{\Om}\setminus\{ x_0\}$, $w(x_0)=0$. Finally, if we choose, say, $\al=(N+1)/(p-1)$, by Example \ref{ex:examples} (iii), for $0<\ve<\min(\ve_0, r_\Ga)$, $w$ is $p$-superharmonious in $\Om$.
\par
Therefore, with these specifications, $w$ is the desired barrier. 
\end{proof}

\smallskip

\section{Limits of variationally $p$-harmonious functions}
\label{sec:limits-harmonious-functions}
\noindent

In this section, we shall prove Theorem \ref{th:limit-of-harmonious}. We will first collect in Section \ref{subsec:geralized-viscosity-solutions} some known facts on the Dirichlet problem \eqref{dirichlet-harmonic}, that we recall here:
$$
\Delta_p^G u=0 \ \mbox{ in } \ \Om,\quad u=g \ \mbox{ on } \ \Ga.
$$

\subsection{Generalized viscosity solutions for \eqref{dirichlet-harmonic}}
\label{subsec:geralized-viscosity-solutions}
We start with the classical definition of a viscosity solution of an elliptic degenerate equation.
Consider a continuous mapping $F : \ol{\Om}\times\RR\times(\RR^N\setminus\{0\})\times{\cS}^N\to\RR$. Here, $\cS^N$ is the set of symmetric $N\times N$ matrices.
The upper and lower semi-continuous envelopes $F^*$ 
and $F_*$ of $F$ are the functions defined by 
$$
F^*(x,s,\xi,X)=\limsup_{(y,t,\eta,Y)\to(x,s,\xi,X)}F(y,t,\eta,Y) 
$$
for $(x,s,\xi,X)\in\ol{\Om}\times\Bbb R\times(\RR^N\setminus\{0\})\times{\cS}^N$, and by $F_*=-(-F)^*$. 
\par
We recall from \cite{CIL, Ko} that a bounded and upper (resp. lower) semi-continuous function 
$u$ is {\it a viscosity subsolution (resp. supersolution)} of $F=0$ in $\ol{\Om}$ if, 
for any $(x,\phi)\in\ol{\Om}\times C^2(\ol{\Om})$ with $\nabla\phi(x)\ne 0$ and such that $u-\phi$ has a local maximum (resp. minimum) at $x$ with $\phi(x)=u(x)$, 
it holds that 
$$
F_*(x,\phi(x),\nabla\phi(x),\nabla^2\phi(x))\le 0 \quad (\mbox{resp. } F^*(x,\phi(x),\nabla\phi(x),\nabla^2\phi(x))\ge 0).
$$
{\it A viscosity solution} of $F=0$ in $\ol{\Om}$ is a function $u\in C(\ol{\Om})$ which is both a viscosity subsolution and supersolution of $F=0$ in $\ol{\Om}$. 
\par
Also, in \cite{BP1, BP2, Is}, the following definitions are considered.
A function $u$ in $\usc(\ol{\Om})\cap L^\infty(\Om)$ (resp. in $\lsc(\ol{\Om})\cap L^\infty(\Om)$) is {\it a generalized viscosity subsolution (resp. supersolution)} of \eqref{dirichlet-harmonic} if, for any $(x,\phi)\in \ol{\Om}\times C^2(\ol{\Om})$ with $\na\phi(x)\ne 0$ and such that
$u-\phi$ has a local maximum (resp. minimum) at $x$ with $u(x)=\phi(x)$, it holds that 
\begin{eqnarray*}
&&-\Delta_p^G\phi(x)\le 0 \ \mbox{ (resp.} -\Delta_p^G\phi(x)\ge 0) \ \hspace{63pt}\mbox{ for } \ x\in\Om, \\
&&\min\left\{
-\Delta_p^G\phi(x),\phi(x)-g(x)
\right\} \le 0 \\ 
&&\hspace{40pt}\mbox{ (resp. } \max\left\{
-\Delta_p^G\phi(x),\phi(x)-g(x)\right\} \ge 0) \ \mbox{ for } \ x\in\Ga.
\end{eqnarray*}
{\it A generalized viscosity solution} of \eqref{dirichlet-harmonic} is a continuous function on $\ol{\Om}$, which is both a generalized viscosity subsolution and supersolution.
This weaker notion of viscosity solution appears naturally as the limiting situation of the dynamic programming principle with respect to \eqref{p-amvp} as is observed below. This formulation has been proved to be equivalent (we refer the reader to \cite{BB,TMP}) to the standard definition of viscosity solution of $F=0$ for the operator $F$ defined by
$$
F(x,s,\xi, X)=-\frac{\tr(X)}{p}-\frac{p-2}{p}\frac{\lan X\xi,\xi\ran}{|\xi|^2} 
$$
for $(x,s,\xi,X)\in\ol{\Om}\times\RR\times(\RR^N\setminus\{0\})\times{\cS}^N$. Here, $\tr(X)$ stands for the trace of the matrix $X$. 
\par
In order to prove Theorem \ref{th:limit-of-harmonious}, we follow an argument used in \cite{BS}. To this aim, we need to set up some further notation.

\begin{prop}
Let $G : \ol{\Om}\times\RR\times(\RR^N\setminus\{0\})\times{\cS}^N\to\RR$ be the mapping defined by 
$$
G(x,s,\xi,X)=\begin{cases}
\displaystyle -\frac{\tr(X)}{p}-\frac{p-2}{p}\frac{\lan X\xi,\xi\ran}{|\xi|^2} \ &\mbox{ if } \ x\in\Om,\\
s-g(x) &\mbox{ if } \ x\in\Ga,
\end{cases}
$$
for $(x,s,\xi,X)\in\ol{\Om}\times\RR\times(\RR^N\setminus\{0\})\times{\cS}^N$. 
Then, we compute:
$$
G^*(x,s,\xi,X)=G_*(x,s,\xi,X)=-\frac{\tr(X)}{p}-\frac{p-2}{p}\frac{\lan X\xi,\xi\ran}{|\xi|^2}  \ \mbox{ if } \ x\in\Om
$$ 
and, if $x\in\Ga$, 
\begin{eqnarray*}
&&G^*(x,s,\xi,X)=\displaystyle\max\left\{
-\frac{\tr(X)}{p}-\frac{p-2}{p}\frac{\lan X\xi,\xi\ran}{|\xi|^2}, s-g(x)
\right\}, \\
&&G_*(x,s,\xi,X)=\displaystyle\min\left\{
-\frac{\tr(X)}{p}-\frac{p-2}{p}\frac{\lan X\xi,\xi\ran}{|\xi|^2}, s-g(x)
\right\}.
\end{eqnarray*}
\end{prop}

\begin{proof}
The first formula for $G^*=G_*=G$ follows from the continuity of $G$ at 
interior points of $\Om$. When $x\in\Ga$ instead, 
we observe that, for every sufficiently small $\de>0$, the supremum 
of $G$ on $(B_\de(x)\cap\ol{\Om})\times B_\de(s)\times B_\de(\xi)\times B_\de(X)$ (where the balls must be intended in the relevant Euclidean spaces) equals
$$
\max\left\{
-\frac{\tr(X)}{p}-\frac{p-2}{p}\frac{\lan X\xi,\xi\ran}{|\xi|^2}, s-g(x)
\right\}. 
$$
The formula for $G^*$ then ensues. The one for $G_*$ easily follows from the formula $G_*=-(-G)^*$. 
\end{proof}

\begin{rem}
{\rm
\label{F^*F_*-computation}
Let $x\in\ol{\Om}$ and $\phi\in C^2(\ol{\Om})$ with $\nabla\phi(x)\ne0$. 
Then, we have that
$$
G(x,\phi(x),\na\phi(x),\nabla^2\phi(x))=
\begin{cases}
-\De^G_p\phi(x) \ &\mbox{ if } \ x\in\Om, \\
\phi(x)-g(x) \ &\mbox{ if } \ x\in\Ga,
\end{cases}
$$
and hence, we obtain for $x\in\Ga$, 
\begin{eqnarray*}
&&G^*(x,\phi(x),\na\phi(x),\na^2\phi(x))=\max\left\{-\De^G_p\phi(x), \,\phi(x)-g(x)\right\},\\
&&G_*(x,\phi(x),\na\phi(x),\na^2\phi(x))=\min\left\{-\De^G_p\phi(x), \,\phi(x)-g(x)\right\}. 
\end{eqnarray*}
}
\end{rem}

By Remark \ref{F^*F_*-computation} 
together with \cite[Theorem 3.3]{TMP}
and the classical weak comparison principle for the viscosity 
solution of \eqref{dirichlet-harmonic} (see \cite{JLM, LM}), we immediately obtain 
the following weak comparison principle 
for $G=0$ in $\ol{\Om}$.

\begin{thm}\label{wcp-p-laplace}
Let $u\in\usc(\ol{\Om})\cap L^\infty(\Om)$ and $v\in\lsc(\ol{\Om})\cap L^\infty(\Om)$.
If $u,v$ are respectively a generalized viscosity subsolution and a generalized viscosity supersolution of $G=0$ in $\ol{\Om}$, then we have that $u\leq v$ on $\ol{\Om}$. 
\end{thm}

\medskip

\subsection{An approximation scheme by $p$-harmonious functions}
We introduce a family 
of mappings ${\cA}_\ve : \RR\times\ol{\Om}\times C(\ol{\Om})\to\RR$ defined for $0<\ve<\ve_0$ by 
$$
{\cA}_\ve(s,x,u)=\begin{cases}
\displaystyle \frac{2(N+p)\ve}{p}\,\frac{s-\mu_p^\ve[u](x)}{r_\ve(x)^2} \ &\mbox{ if } \ x\in\Om,
\\
\ve\,[s-g(x)] \ &\mbox{ if } \ x\in\Ga 
\end{cases}
$$
for $(s,x,u)\in\Bbb R\times\ol{\Om}\times C(\ol{\Om})$.
Here, $\ve_0>0$ is taken so small so as to be sure that 
the set $\Om_\ve$ is 
still a domain for every $0<\ve<\ve_0$. 

\begin{lem}
\label{consistency-p}
If $x\in\ol{\Om}$ and $\phi\in C^2(\ol{\Om})$ is such that 
$\nabla\phi(x)\ne 0$, then we have that
\begin{eqnarray*}
&&\displaystyle\limsup_{(\ve, y, \de)\to(0^+, x, 0)}\frac{
{\cA}_\ve(\phi(y)+\de, y, \phi+\de)
}{\ve}=G^*(x,\phi(x),\nabla \phi(x),\nabla^2\phi(x)),\\
&&\displaystyle\liminf_{(\ve, y, \de)\to(0^+, x, 0)}\frac{
{\cA}_\ve(\phi(y)+\de, y, \phi+\de)
}{\ve}=G_*(x,\phi(x),\nabla \phi(x),\nabla^2\phi(x)). 
\end{eqnarray*}
\end{lem}

\begin{proof}
By using elementary properties of $p$-means (see \cite[Proposition 2.7]{IMW}), for every $\de\in\RR$, we have that 
\begin{multline*}
\frac{{\cA}_\ve(\phi(y)+\de, y, \phi+\de)}{\ve}= \\
\frac{2(N+p)}{p}\,\frac{\phi(y)+\de-\mu_p^\ve[\phi+\de](y)}{r_\ve(y)^2}= 
\frac{2(N+p)}{p}\,\frac{\phi(y)-\mu_p^\ve[\phi](y)}{r_\ve(y)^2}
\ \mbox{ if } \ y\in\Om,
\end{multline*}
and 
$$
\frac{{\cA}_\ve(\phi(y)+\de, y, \phi+\de)}{\ve}
=\phi(y)+\de-g(y)\quad\text{if \,}y\in\Ga. 
$$
Then if $x\in\Omega$, by the uniform convergence claimed in Lemma \ref{uni-amvp} 
and the continuity of $\phi$ and $g$, we then simply infer that 
\begin{multline*}
\limsup_{(\ve, y, \de)\to(0^+, x, 0)}\frac{
{\cA}_\ve(\phi(y)+\de, y, \phi+\de)
}{\ve}= \\
\begin{cases}
-\Delta^G_p\phi(x) \ &\mbox{ if } \ x\in\Om, \\
\max\left\{
-\De^G_p\phi(x), \phi(x)-g(x)
\right\} \ &\mbox{ if } \ x\in\Ga.
\end{cases}
\end{multline*}
The first claim is thus proved. The second claim follows similarly. 
\end{proof}

We are now ready to prove our final result. 

\begin{proof}[\bf Proof of Theorem \ref{th:limit-of-harmonious}]
Up to re-defining $\ve_0$, Theorem \ref{thm:existence-uniqueness-harmonious} tells us that, for every $\ve\in (0,\ve_0)$, there exists a unique function $u^\ve\in C(\ol{\Om})$ that solves the equation 
\begin{equation}
\label{equation-A-eps}
{\cA}_\ve(u^\ve(x),x,u^\ve)=0 \ \mbox{  for } \ x\in\ol{\Om}.
\end{equation}
Also, by Lemma \ref{lem3} (i), 
we have that $u^\ve$ is bounded in $\ol{\Om}$ from below and above 
by the minimum and the maximum of $g$ on $\Ga$, and hence 
independently of $\ve$. Thus, by this property, the functions defined for 
$x\in\ol{\Om}$ by 
$$
u_*(x)=\liminf_{(\ve,y)\to(0^+,x)}u^\ve(y)
\quad\mbox{and}\quad
u^*(x)=\limsup_{(\ve,y)\to(0^+,x)}u^\ve(y)
$$
are bounded in $\ol{\Om}$. Also, $u_*$ and $u^*$ are lower semi-continuous and, respectively, upper semi-continuous in $\ol{\Om}$, by a standard result. 
\par
The strategy of the proof is to show that $u_*$ and $u^*$ are a viscosity supersolution 
and a viscosity subsolution of $G=0$ in $\ol{\Om}$. In fact, if we prove that, 
the comparison principle in Theorem \ref{wcp-p-laplace} would give that $u^*\le u_*$ on $\ol{\Om}$. Thus, since $u_*\le u^*$ on $\ol{\Om}$ by construction, we would infer that $u_*\equiv u^*$ on $\ol{\Om}$. As a consequence, the function defined on $\ol{\Om}$ by 
$u=u_*\equiv u^*$ would be continuous and a viscosity solution of $G=0$ on $\ol{\Om}$. Then, $u$ would be the unique viscosity solution of \eqref{dirichlet-harmonic} 
by \cite[Theorem 3.3]{TMP} and Remark \ref{F^*F_*-computation}. 
\par
We also point out that, while $u^\ve$ converges pointwise to $u$ by construction, 
we can further prove that this convergence is uniform on $\ol{\Om}$. In fact, 
$u$ turns out to be the limit of the monotonic sequences defined 
for $k\in\NN$ and $x\in\ol{\Om}$ by 
\begin{eqnarray*}
&&v_k(x)=\sup\Bigl\{
u^\ve(y)\,|\,0<\ve<1/k, \,|y-x|<1/k, \,y\in\ol{\Om}
\Bigr\},\vspace{2pt} \\
&&w_k(x)=\inf\Bigl\{
u^\ve(y)\,|\,0<\ve<1/k, \,|y-x|<1/k, \,y\in\ol{\Om}
\Bigr\}. 
\end{eqnarray*}
The uniform convergence on $\ol{\Om}$ of these sequences follows from Dini's monotone convergence theorem (see \cite[Theorem 7.13]{Ru}), 
since they are monotonic,  $\ol{\Om}$ is compact, $v_k$ is lower semicontinuous, $w_k$ is upper semicontinuous, and $u$ is continuous on $\ol{\Om}$. The uniform convergence of $u^\ve$ then follows by 
observing that 
$$
w_k(x)\le u^\ve(x)\le v_k(x) \ \mbox{ for  } \ x\in\ol{\Om}
\quad\mbox{and}\quad 0<\ve<1/k. 
$$

Therefore, to complete the proof, we are left to show that $u_*$ and 
$u^*$ are a viscosity supersolution and a viscosity subsolution of $G=0$ in $\ol{\Om}$. We shall prove that fact only for $u^*$, since 
for $u_*$ we can proceed similarly. 
\par
We preliminarily notice that, 
by \cite[Theorem 2.5]{IMW}, the mapping ${\cA}_\ve$ is decreasing 
in the third variable, in the sense that for $(s,x)\in\Bbb R\times\ol{\Om}$ and $u,v\in C(\ol{\Om})$, it holds
$$
{\cA}_\ve(s,x,u)\geq{\cA}_\ve(s,x,v)
\quad\text{if \,}u\leq v\quad\text{in \,}\ol{\Om}. 
$$

Now, let $(x,\phi)\in\ol{\Om}\times C^2(\ol{\Om})$ with $\nabla\phi(x)\ne0$ be such that $u^*-\phi$ has a local maximum at $x$ with $u^*(x)=\phi(x)$. 
Without loss of generality, we can always suppose that the maximum is global and strict, 
that is $u^*-\phi<u^*(x)-\phi(x)=0$ in $\ol{\Om}\setminus\{x\}$. 
By a standard argument in the theory of viscosity solutions (see \cite{Ko}), 
we know that there exists a sequence of elements $(\ve_j,x_j)\in (0,\ve_0)\times\ol{\Om}$ such that,
for each fixed $j\in\NN$,  $x_j$ is a global maximum point 
for $u^{\ve_j}-\phi$ and
$$
(\ve_j, x_j, u^{\ve_j}(x_j))\to (0,x, u^*(x)) \ \mbox{ as } \ j\to\infty.
$$
If we set $\de_j=u^{\ve_j}(x_j)-\phi(x_j)$, then $\de_j\to 0$ as $j\to\infty$ and $u^{\ve_j}-\phi\le\de_j$ on $\ol{\Om}$. 
\par
Now, we use \eqref{equation-A-eps} at $x_j$ for $\ve=\ve_j$ and $u=u^{\ve_j}$, and infer that 
\begin{multline*}
0={\cA}_{\ve_j}(u^{\ve_j}(x_j),x_j, u^{\ve_j})= \\
{\cA}_{\ve_j}(\phi(x_j)+\de_j, x_j, u^{\ve_j})
\ge{\cA}_{\ve_j}(\phi(x_j)+\de_j, x_j, \phi+\de_j).
\end{multline*}
Here, the last inequality follows from the aforementioned monotonicity of ${\cA}_{\ve_j}$. Then since $\nabla\phi(x)\ne0$, we can apply Lemma \ref{consistency-p} 
to obtain that 
\begin{multline*}
0\geq\liminf_{j\to\infty}\frac{{\cA}_{\ve_j}(\phi(x_j)+\de_j,x_j,\phi+\de_j)}{\ve_j}\ge \\
\liminf_{(\ve,y,\de)\to(0^+,x,0)}\frac{{\cA}_\ve(\phi(y)+\de,y,\phi+\de)}{\ve}=F_*(x,\phi(x),\na\phi(x),\na^2\phi(x)). 
\end{multline*}
Hence, $u^*$ is a viscosity subsolution of $G=0$ on $\ol{\Om}$. 
The proof of Theorem \ref{th:limit-of-harmonious} is complete. 
\end{proof}

\section*{Acknowledgements}
The third author was partially supported by the Gruppo Nazionale di Analisi Matematica, Probabilit\`{a} e
Applicazioni (GNAMPA) of the Istituto Nazionale di Alta Matematica (INdAM). Part of the research related to this paper was carried out during a his visit to Kanazawa University and Osaka University. He wishes to thank those institutions for their warm hospitality.

\end{document}